\documentclass[preprint,12pt]{elsarticle}

\usepackage{amssymb}

\def \N{\mathbb N}
\journal{Optimization Letters}

\usepackage{mathdots}
\usepackage{amsfonts}
\usepackage{amssymb}
\usepackage{amsthm}
\usepackage[toc, page]{appendix}
\usepackage{latexsym,amsmath}
\usepackage{bm}
\usepackage{enumitem}
\usepackage{graphicx}
\usepackage{afterpage}
\usepackage{hyperref}
\usepackage{mathrsfs}
\usepackage{mathtools}
\usepackage{pzccal}
\usepackage{subcaption}
\usepackage{subfiles}	
\usepackage[short]{optidef}
\usepackage{tabu}
\usepackage{titlesec}
\usepackage{verbatim}
\usepackage{cleveref}
\usepackage{xpatch}
\usepackage[table]{xcolor}
\usepackage{float}
\usepackage{array}
\usepackage[skip=20pt, position=top]{caption}
\usepackage{float}
\usepackage[english]{babel}
\usepackage{blindtext}
\restylefloat{table}

\newcommand\Tstrut{\rule{0pt}{2.6ex}}
\newcommand\Bstrut{\rule[-0.9ex]{0pt}{0pt}}
\newcommand{\vect}[1]{\mbox{\boldmath $ #1 $}}
\newcommand{\mymat}[2]{\left[\begin{array}{#1} #2 \end{array}\right]}
\newcommand{\under}[1]{\begin{footnotesize}\text{$#1$}\end{footnotesize}}
\newcommand{\ZZ}{\mathcal{Z}}
\newcommand{\rS}{\mathbb{S}}
\def\norm#1{\left \Vert #1 \right \Vert}
\newcommand{\R}{\mathbb{R}}
\newtheorem{theorem}{Theorem}[section]
\newtheorem{lemma}[theorem]{Lemma}
\newtheorem{proposition}[theorem]{Proposition}
\newtheorem{corollary}[theorem]{Corollary}

\newtheorem{definition}[theorem]{Definition}

\newtheorem{remark}[theorem]{Remark}

\numberwithin{equation}{section}
\newcommand{\ds}{\displaystyle}
\newcommand{\name}[1]{\normalfont\text{(#1)}\quad}

\begin{document}

\begin{frontmatter}



\title{\large {\textbf{\textsf{Exact Conic Programming Reformulations of Two-Stage Adjustable Robust Linear Programs with New Quadratic Decision Rules}}}}


\author{D. Woolnough, V. Jeyakumar\footnote{Corresponding author: V. Jeyakumer, Department of Applied Mathematics, UNSW Sydney, Australia. Email: v.jeyakumar@unsw.edu.au} and G. Li}

\address{Department of Applied Mathematics, UNSW Sydney, Australia}

\address{{\bf Revised Version: \today}}
\begin{abstract}
In this paper we introduce a new parameterized Quadratic Decision Rule (QDR), a generalisation of the commonly employed Affine Decision Rule (ADR), for two-stage linear adjustable robust optimization problems with ellipsoidal uncertainty and show that (affinely parameterized) linear adjustable robust optimization problems with QDRs are numerically tractable by presenting exact  semi-definite program (SDP) and second order cone program (SOCP) reformulations. Under these QDRs, we also establish that exact conic program  reformulations also hold for two-stage linear ARO problems, containing also adjustable variables in their objective functions. We then show via numerical experiments on  lot-sizing problems with uncertain demand that  adjustable robust linear optimization problems with QDRs improve upon the ADRs in their performance both in the worst-case sense and after simulated realization of the uncertain demand relative to the true solution.

\end{abstract}



\begin{keyword}
Adjustable robust optimization \sep semi-definite programs\sep second order cone programs \sep ellipsoidal uncertainty \sep robust linear optimization



\end{keyword}

\end{frontmatter}


\section{Introduction }
Consider the two-stage linear Adjustable Robust Optimization (ARO) problem with an ellipsoidal uncertainty set
\begin{mini}
{\substack{\under{\vect{x},\vect{y}(\cdot)}}}{\vect{c}^T\vect{x}}{}{\name{$P_0$}}
	\addConstraint{A(\vect{z})\vect{x} + B\vect{y}(\vect{z})}{\leq\vect{d}(\vect{z}),\quad}{\forall\vect{z}\in\ZZ}
\end{mini}
where $\ZZ = \left\{\vect{z}\in\R^l : \norm{\vect{z}}^2 \leq r^2, r > 0\right\}$ is the user specified ellipsoidal uncertainty set,  $\vect{x}\in \R^n$ is the first-stage ``here and now"  decision that is made before $\vect{z}\in\R^l$ is realized, $\vect{y}(\vect{z})\in\R^k$ is the second-stage ``wait and see" decision that can be adjusted according to  the actual data; the coefficient matrix $A\in \R^{m\times n}$ and the right hand side vector $\vect{d}\in\R^m$ depend on the uncertainty parameter $\vect{z}$, and the (fixed recourse) coefficient matrix $B=(\vect{b}_1, \ldots, \vect{b}_m)^T, \; \vect{b}_i\in\R^k $ does not depend on $\vect{z}$.

The ARO approach, which employs ARO model problems  of the form $(P_0)$, is  less conservative than the traditional Robust Optimization (RO) methodology, pioneered by Ben-tal et. al \cite{robustbook,7,34,siamRevRobust,Goberna-Jeya-Li15,jeya-optimL}, as  it yields more flexible decisions that can be adjusted according to the realized portion of data at a given stage, and so allows multi-stage decision-making in practical applications \cite{De-tutorial15}. Moreover,  ARO provides optimal objective values that are at least as good as that of the standard RO
approach \cite{robustbook,dendick-18}.

However, the two-stage ARO problem $(P_0)$ is a challenging optimization problem to study, theoretically and numerically  because a linear function is optimized over $\vect{y}(\cdot)$, which are mappings $\vect{y}:\ZZ\rightarrow \R^k$, rather than vectors. It is generally hard to obtain a numerically tractable characterization of the system with a mapping $\vect{y}(\cdot)$ unless the mapping is restricted to satisfy some special rules, called ``decision rules". Traditionally, $\vect{y}(\cdot)$ is assumed to satisfy an Affine Decision Rule (ADR), such as
$\vect{y}(\vect{z})= \vect{y}_0 + W\vect{z}, $
where  $ \vect{y}_0\in \R^k, \;  W\in \R^{k\times l}$ are the coefficients of the decision rule that are to be optimized \cite{robustbook,Chen-09-OR}.

In many cases, affine decision rules, in particular, for the affinely parameterized ARO problems \cite{Ben-tal-MP-04,Survey-19}, often result in  computationally tractable reformulations and have been known to give optimal or near optimal solutions for broad classes of practical problems,, e.g. inventory management \cite{robustbook}. On the other hand, transformations of two-stage ARO problems with nonlinear decision rules to single-stage robust problems often result in hard non-convex optimization problems \cite{robustbook}. Consequently, the study of computational tractability and applicability of these problems with nonlinear decision rules is of great interest in robust optimization.

In this paper we examine {\it affinely parameterized two-stage adjustable robust linear optimization problems} with quadratic decision rules under an ellipsoidal uncertainty set and make the following contributions.
\begin{itemize}
\item[(i)] We introduce a new parameterized Quadratic Decision Rule (QDR), generalizing the commonly employed affine decision rule, and show that affinely parameterized linear ARO problems with QDRs are numerically tractable by presenting exact conic reformulations. In particular, we establish exact second order cone program (SOCP) reformulations for the linear ARO problems under a special separable QDRs.

    We do this by generalizing the approach of \cite{34,robustbook,jeyaChuong} for ADRs and employing the $\mathcal{S}$-lemma \cite{Ben-tal01} and the Schur's complement.
    We further show how exact conic programming reformulations can be derived from our results for ARO problems with adjustable variables also in their objective functions as they appear in many practical decision-making models of optimization, such as the lot-sizing problem with uncertain demand.

    Various  nonlinear decision rules, such as the homogeneous \cite{arxiv} and non-homogeneous quadratic decision rules \cite{thesis,robustbook}, and polynomial decision rules \cite{bert}, have also recently been used to approximate and reformulate ARO problems. Our results readily yield corresponding exact conic program reformulations for affinely parameterized linear ARO problems \cite{Ben-tal-MP-04} with affine decision rules and homogeneous as well as non-homogeneous quadratic decision rules.

    \item[(ii)] We employ our SDP and SOCP reformulations to solve the lot-sizing problem with uncertain demand and present a comparison of our techniques in their performance by contrasting their optimal solutions both in the worst-case sense and after simulated realisations of the uncertain demand. Numerical experiments on lot-sizing problems demonstrate that the quadratic decision rule outperforms affine decision rules in both cases, whilst the time taken to solve problems with quadratic decision rules is significantly greater (due to the larger number of variables) than the ones with affine decision rules.

        \end{itemize}


\medskip

In section 2 we present the parameterized quadratic decision rule, an extension of the affine decision rule, and present exact SDP and SOCP reformulations for two-stage ARO problems. In section 3, we derive exact conic programming reformulations for ARO problems with adjustable variables also in their objective functions. In section 4, we employ our reformulation schemes to solve the lot-sizing problem and show that it is both consistent with the ADR and improves upon it. In section 5 we present concluding remarks with a brief discussion on further research.
\medskip

\section{Quadratic Decision Rules \& Exact Conic Program Reformulations}

We begin by fixing some preliminaries. The notation $\R^n$ signifies  the Euclidean  space for each $n\in\N:=\{1,2,\ldots\}$ and $\rS_l$ is the space of all real $l \times l$ symmetric matrices. As usual, the symbol $I_n$ stands for the identity $(n\times n)$ matrix, while  $\R_+:=[0,+\infty)\subset \R.$ The inner product in $\R^n$ is defined by $\langle x,y\rangle:=x^T y$ for all $x, y\in\R^n.$  A symmetric $(n\times n)$  matrix $A$ is said to be  positive semi-definite, denoted by $A\succeq 0$, whenever $x^T Ax\ge 0$ for all $x\in\R^n.$
\medskip

In this section, we present numerically tractable conic linear program reformulations of the affinely adjustable case of the two-stage robust linear optimization problem $(P_0)$ under a  parameterised quadratic decision rule (QDR) which is defined as follows:
\begin{definition}[{\bf Quadratic Decision Rule}] Let $\theta\in [0,1]$. The ARO problem $(P_0)$ is said to satisfy the parameterized quadratic decision rule whenever the mapping $\vect{y}(\cdot)$ is restricted to mappings of the form
\[
\vect{y}(\vect{z}) = \theta(\vect{y}_0 + W\vect{z}) + (1-\theta) \mymat{c}{\vect{z}^TQ_1\vect{z} \\ \vect{z}^T Q_2\vect{z} \\ \vdots \\ \vect{z}^T Q_k \vect{z}}.
\]
\end{definition}
We define the following operator to simplify working:
\[
\vect{z}^T \mathcal{Q}_k \vect{z} = \mymat{c}{\vect{z}^T Q_1 \vect{z} \\ \vdots \\ \vect{z}^T Q_k \vect{z}}
\]
so that our QDR is $\vect{y}(\vect{z}) = \theta(\vect{y}_0 + W\vect{z}) + (1-\theta)\vect{z}^T\mathcal{Q}_k \vect{z}$.

\medskip


\noindent\textbf{QDRs and SDP Reformulations}. Consider the following affinely parameterized version of ARO problem $(P_0)$ with the parameterized QDR,
\begin{mini}
	{\substack{\under{\vect{x},\vect{y}_0},\\W,Q_j,j=1,\dots,k}}{\vect{c}^T\vect{x}}{}{(P)}
	\addConstraint{A(\vect{z})\vect{x} + B\left(\theta(\vect{y}_0 + W\vect{z}) + (1-\theta)\vect{z}^T \mathcal{Q}_k \vect{z}\right)}{\leq\vect{d}(\vect{z}),\quad}{\forall\vect{z}\in\ZZ,}
\end{mini}
where $\ZZ = \left\{\vect{z}\in\R^l : \norm{\vect{z}}^2 \leq r^2\right\}$ is ellipsoidal uncertainty set; $\vect{c}\in\R^n$; $B=(\vect{b}_1, \ldots, \vect{b}_m)^T; \vect{b}\in\R^k $;
$A(\vect{z}) = (\vect{a}_1+A_1\vect{z}, \dots, \vect{a}_m + A_m\vect{z})^T,\; \vect{a}_i\in\R^{n},\; A_i\in\R^{n\times l}$, $\vect{d}(\vect{z}) = (d_{0,1}+\vect{d}_1^T\vect{z},\dots,d_{0,m} + \vect{d}_m^T\vect{z})^T,\; d_{0,i}\in\R$, $\vect{d}_i\in\R^{l}$ and $\theta \in [0, 1]$.

We associate with ($P$) the following semi-definite program

\begingroup
\begin{mini*}
{\substack{\under{\vect{x},\vect{y}_0, \vect{\lambda}},\\W,Q_j,j=1,\dots,k}}{\vect{c}^T\vect{x}}{}{(P-QDR)}
	\addConstraint{ \lambda_i\geq 0, \; i=1,\dots,m}, \mymat{ccc}{
	P_1 & &  \\
	& \ddots \\
	& & P_m
	}\succeq 0,
\end{mini*}
where $\vect{x}\in\R^{n},\;\vect{y}_0\in\R^{k},\;W\in\R^{k\times l},\; Q_j\in\rS_l,j=1,\dots,k$ and
\[
P_i = \mymat{cc}{d_{0,i}-\vect{a}_i^T\vect{x}-\theta\vect{b}_i^T\vect{y}_0-\lambda_i r^2 & \ds\frac{1}{2}(\vect{d}_i^T-\vect{x}^T A_i-\theta\vect{b}_i^T W) \\ \ds\frac{1}{2}(\vect{d}_i^T-\vect{x}^T A_i-\theta\vect{b}_i^T W)^T & \lambda_i I_l -(1-\theta)\ds\sum_{j=1}^k{(\vect{b}_i)_j Q_j}},\quad i = 1,\dots,m.
\]
\endgroup

We first show that the problem (P) admits an exact SDP reformulation in the sense that the objective values of (P) and (P-QDR) are equal and their constraint systems are equivalent.
To do this, we first recall the celebrated $\mathcal{S}$-Lemma \cite{Ben-tal01} which is a useful tool in nonconvex quadratic optimization.
\begin{lemma}[$\mathcal{S}$-Lemma]
Let $A$, $B$ be two symmetric matrices such that there exists a $\vect{z}_0$ such that $\vect{z}_0^T A \vect{z}_0> 0$. Then,  
\[
\vect{z}^T A \vect{z}\geq 0 \implies \vect{z}^T B \vect{z}\geq 0
\]
holds true if and only if
\[
\exists\lambda\geq 0 : B-\lambda A\succeq 0.
\]
\end{lemma}

The following Theorem provides an exact SDP reformulation result for the linear ARO problem (P).
\begin{theorem}[{\bf General QDRs and Exact SDP Reformulations}] \label{thm:1}
Let $\theta \in [0, 1]$. Consider the linear ARO problem (P) with the parameterized quadratic decision rule and its associated semi-definite program (P-QDR). Then, problem (P) and the semi-definite program (P-QDR) are equivalent, in
the sense that, $(\vect{x},\vect{y}_0,W,Q_1,\ldots,Q_k)$  is a solution for (P) if and only if there exists $\vect{\lambda} \in \mathbb{R}^m_+$ such that
$(\vect{x},\vect{y}_0,\vect{\lambda},W,Q_1,\ldots,Q_k)$  is a solution for (P-QDR). Moreover,
$
\min{\text{\emph{(P)}}} = \min{\text{\emph{(P-QDR)}}}.
$


\end{theorem}
\begin{proof} The constraint system of (P)
\[
A(\vect{z})\vect{x} + B\left(\theta(\vect{y}_0 + W\vect{z}) + (1-\theta)\vect{z}^T \mathcal{Q}_k\vect{z}\right)\leq\vect{d}(\vect{z}),\quad \forall\vect{z}\in\ZZ
\]
is equivalently re-written as the following semi-infinite system of $m$ constraints:
\begin{equation}\label{eq:mconstraintsA}
 (\vect{a}_i + A_i\vect{z})^T\vect{x} + \vect{b}_i^T\left(\theta(\vect{y}_0 + W\vect{z}) + (1-\theta)\vect{z}^T \mathcal{Q}_k \vect{z}\right) \leq d_{0,i} + \vect{d}_i^T\vect{z},\;\forall\vect{z}\in\ZZ, \; i=1,2,\ldots, m
\end{equation}
For each $i=1,2,\ldots, m$, we claim that the system
\[
(\vect{a}_i + A_i\vect{z})^T\vect{x} + \vect{b}_i^T\left(\theta(\vect{y}_0 + W\vect{z}) + (1-\theta)\vect{z}^T \mathcal{Q}_k \vect{z}\right) \leq d_{0,i} + \vect{d}_i^T\vect{z},\;\forall\vect{z}\in\ZZ
\]
is equivalent to the liner matrix inequality:
\begin{equation}\label{eq:LMI}
 \exists \lambda_i \geq 0,  \; \;  \; \;
\mymat{cc}{d_{0,i}-\vect{a}_i^T\vect{x}-\theta\vect{b}_i^T\vect{y}_0-\lambda_i r^2 & \ds\frac{1}{2}(\vect{d}_i^T-\vect{x}^T A_i-\theta\vect{b}_i^T W) \\ \ds\frac{1}{2}(\vect{d}_i^T-\vect{x}^T A_i-\theta\vect{b}_i^T W)^T & \lambda_i I_l -(1-\theta)\ds\sum_{j=1}^k{(\vect{b}_i)_j Q_j}}\succeq 0.
\end{equation}
Granting this, we obtain that  $(\vect{x},\vect{y}_0,W,Q_j) \in \mathbb{R}^n \times \mathbb{R}^{k} \times \mathbb{R}^{k \times l} \times \mathbb{S}^l$, $j=1,\dots,k$, satisfies the system of constraints in \eqref{eq:mconstraintsA} if and only if
there exists $\vect{\lambda} \in \mathbb{R}^m_+$ such that $(\vect{x},\vect{y}_0,\vect{\lambda},W,Q_j) \in \mathbb{R}^n \times \mathbb{R}^{k} \times \mathbb{R}^m \times \mathbb{R}^{k \times l} \times \mathbb{S}^l$
satisfies  the semi-definite constraint system of (P-QDR).  As the objective functions of both problems (P) and (P-QDR) are the same, we see that
problem (P) and the semi-definite program (P-QDR) are equivalent and
 $\min{\text{(P)}} = \min{\text{(P-QDR)}}$. Then, the conclusion of this theorem follows.

We now turn to the proof of the claim. Fix $i \in \{1,\ldots,m\}$.
Then,
\begingroup
\allowdisplaybreaks
\begin{eqnarray*}
 & & (\vect{a}_i + A_i\vect{z})^T\vect{x} + \vect{b}_i^T\left(\theta(\vect{y}_0 + W\vect{z}) + (1-\theta)\vect{z}^T \mathcal{Q}_k \vect{z}\right) \leq d_{0,i} + \vect{d}_i^T\vect{z},\;\forall\vect{z}\in\ZZ \\
& \iff & (\vect{a}_i + A_i\vect{z})^T\vect{x} + \vect{b}_i^T\left(\theta(\vect{y}_0 + W\vect{z}) + (1-\theta)\mymat{c}{\vect{z}^TQ_1\vect{z} \\ \vdots \\ \vect{z}^T Q_k \vect{z}}\right) \leq d_{0,i} + \vect{d}_i^T\vect{z},\quad\forall\vect{z}\in\ZZ \\
&	\iff & (d_{0,i}-\vect{a}_i^T\vect{x}-\theta\vect{b}_i^T\vect{y}_0) + \left(\vect{d}_i^T-\vect{x}^T A_i-\theta\vect{b}_i^T W\right)\vect{z}- \vect{z}^T\left((1-\theta)\ds\sum_{j=1}^k{(\vect{b}_i)_j Q_j}\right)\vect{z}\geq 0,  \forall\vect{z}\in\ZZ
\end{eqnarray*}
\endgroup
which is, in turn, equivalent to the implication:
\begin{align}
\begin{split}
r^2-\vect{z}^T\vect{z}\geq 0 \implies (d_{0,i}-\vect{a}_i^T\vect{x}-\theta\vect{b}_i^T\vect{y}_0) + \left(\vect{d}_i^T-\vect{x}^T A_i-\theta\vect{b}_i^T W\right)\vect{z} \\
 -\vect{z}^T\left((1-\theta)\ds\sum_{j=1}^k{(\vect{b}_i)_j Q_j}\right)\vect{z}\geq 0.
\end{split}
\end{align}
Letting $\vect{u} = \mymat{c}{1\\\vect{z}}$ we can write the above implication  as
\[
\vect{u}^T P \vect{u}\geq 0 \implies \vect{u}^T R_i \vect{u}\geq 0,
\]
where
\[
P = \mymat{cc}{r^2 & 0\\ 0 & -I_l},\quad R_i =
\mymat{cc}{d_{0,i}-\vect{a}_i^T\vect{x}-\theta\vect{b}_i^T\vect{y}_0 & \ds\frac{1}{2}(\vect{d}_i^T-\vect{x}^T A_i-\theta\vect{b}_i^T W) \\
\ds\frac{1}{2}(\vect{d}_i^T-\vect{x}^T A_i-\theta\vect{b}_i^T W)^T &  -(1-\theta)\ds\sum_{j=1}^k{(\vect{b}_i)_j Q_j}}.
\]
Clearly $P$ and $R$ are symmetric matrices. If we choose $\vect{u}_0 = \mymat{cc}{1 & \vect{0}}^T$ then $\vect{u}_0^T P \vect{u}_0 = r^2 > 0$ and so the $\mathcal{S}$-Lemma \cite{Ben-tal01} applies. Hence, \eqref{eq:mconstraintsA} is equivalent to the linear matrix inequality:
\begingroup
\allowdisplaybreaks
\begin{align*}
& \exists\lambda_i \geq 0 : R_i-\lambda_i P\succeq 0 \\
& \iff \lambda_i\geq 0, \mymat{cc}{d_{0,i}-\vect{a}_i^T\vect{x}-\theta\vect{b}_i^T\vect{y}_0-\lambda_i r^2 & \ds\frac{1}{2}(\vect{d}_i^T-\vect{x}^T A_i-\theta\vect{b}_i^T W) \\ \ds\frac{1}{2}(\vect{d}_i^T-\vect{x}^T A_i-\theta\vect{b}_i^T W)^T & \lambda_i I_l -(1-\theta)\ds\sum_{j=1}^k{(\vect{b}_i)_j Q_j}}\succeq 0.
\end{align*}
\endgroup
Thus, the claim follows.
\end{proof}
\begin{remark}[{\bf Exact SDPs for affine \& and other known quadratic decision rules}] It is worth noting that Theorem 2.3. readily yields exact SDP reformulations for linear ARO problems with affine decision rules \cite{robustbook} by setting $\theta =1$, homogeneous quadratic decision rules \cite{arxiv} by setting $\theta =0$ and with non homogeneous quadratic decision rules \cite{thesis} by setting $\theta =\frac{1}{2}$.
\end{remark}
\medskip

\noindent\textbf{Separable QDRs and SOCP Reformulations}. We now show that, if we consider a restricted version of the quadratic decision rule (see Definition 2.1), then the ADR problem can be equivalently reformulated as a second order cone programming problem. Second order cone programming reformulations for classes of nonconvex quadratic optimization problems and robust optimization problems have been of great interest in recent years   \cite{bental1,JL1}. This is because the second order cone programming method has proved to be a powerful scheme for solving various class of practical optimization problems and advanced commercial software is available to solve SOCPs.

\begin{definition}[{\bf Separable Quadratic Decision Rule}]\label{def2} Let $\theta\in [0,1]$. The ARO problem $(P_0)$ is said to satisfy the parameterized separable quadratic decision rule whenever the mapping $\vect{y}(\cdot)$ is restricted to mappings of the form
\[
\vect{y}(\vect{z}) = \theta(\vect{y}_0 + W\vect{z}) + (1-\theta) \mymat{c}{\vect{z}^TQ_1\vect{z} \\ \vect{z}^T Q_2\vect{z} \\ \vdots \\ \vect{z}^T Q_k \vect{z}}=\theta(\vect{y}_0 + W\vect{z}) + (1-\theta) \mymat{c}{\ds\sum_{p=1}^l q_{1,p}z_p^2 \\ \ds \sum_{p=1}^l q_{2,p}z_p^2 \\ \vdots \\ \ds \sum_{p=1}^l q_{k,p}z_p^2 }, 
\]
where $Q_j$, $j=1,\ldots,k$, are diagonal matrices whose diagonal elements are $q_{1,j},\ldots,q_{l,j}$.

\end{definition}

We now consider the following affinely parameterized version of ARO problem $(P_0)$ with the separable quadratic decision rule:

\begin{mini}
	{\substack{\under{\vect{x},\vect{y}_0},\\W,Q_j,j=1,\dots,k}}{\vect{c}^T\vect{x}}{}{\name{$P_s$}}
	\addConstraint{A(\vect{z})\vect{x} + B\left(\theta(\vect{y}_0 + W\vect{z}) + (1-\theta)\vect{z}^T \mathcal{Q}_k \vect{z}\right)}{\leq\vect{d}(\vect{z}),\quad}{\forall\vect{z}\in\ZZ,}
\end{mini}
where  $\vect{z}^T \mathcal{Q}_k \vect{z} = \mymat{c}{\vect{z}^T Q_1 \vect{z} \\ \vdots \\ \vect{z}^T Q_k \vect{z}}$ 
and each $Q_j$, $j=1,\ldots,k$, is a diagonal matrix whose diagonal elements are $q_{1,j},\ldots,q_{l,j}$. Other assumptions on $(P_s)$ are the same as on $(P)$.

To do this, we first show that, using a linear transform and Schur's complement, the constraints of ($P_s$) (with separable quadratic decision rule) can be characterized in terms of second
order cone constraints.
\begin{proposition}[{\bf Equivalent Second-order Cone Constraints}] \label{lemma:2.3}
Let $\theta \in [0, 1]$; let $\vect{a}\in \R^n, \; A\in \R^{n\times l}, \; \vect{b}\in \R^k, \; \vect{d}\in\R^l, \; d_0\in\R$; let $\vect{x}\in\R^{n},\;\vect{y}_0\in\R^{k},\; W\in\R^{k\times l},\; Q_j\in\rS_l,\; j=1,\dots,k$.
Let $\ZZ$ be an ellipsoidal uncertainty set, defined by $\ZZ = \left\{\vect{z}\in \R^l : \norm{\vect{z}}^2\leq r^2\right\}$.
Suppose that each $Q_j$, $j=1,\ldots,k$, is a diagonal matrix whose diagonal elements are $q_{1,j},\ldots,q_{l,j}$.
Then, the following systems are equivalent:
\begin{enumerate}
\item[{\rm (I)}] $
(\vect{a} + A\vect{z})^T\vect{x} + \vect{b}^T\left(\theta(\vect{y}_0 + W\vect{z}) + (1-\theta)\vect{z}^T \mathcal{Q}_k \vect{z}\right) \leq d_0 + \vect{d}^T\vect{z},\quad\forall\vect{z}\in\ZZ
$
\item[{\rm (II)}] There exist $\lambda \in \R_+$  and $s_p \in \R_+$, $p=1,\ldots,l$, such that
\[
\left\{ \begin{array}{l}
         \displaystyle \sum_{p=1}^l s_p \le d_0-\vect{a}^T\vect{x}-\theta\vect{b}^T\vect{y}_0-\lambda r^2, \\
         \lambda  -(1-\theta)\sigma_p \ge 0, \ \ \ p=1,\ldots,l, \\
          \left\|\left((\vect{d}-A^T\vect{x} -\theta W^T\vect{b})_p,s_p-\lambda  +(1-\theta)\sigma_p  \right) \right)\| \le s_p +\lambda  -(1-\theta)\sigma_p,\ \   \ p=1,\ldots,l.
        \end{array}
\right.
\]
{\rm Here, $\sigma_p=\sum_{j=1}^k b_j q_{p,j}$, $p=1,\ldots,l$ are the diagonal elements of $\sum_{j=1}^k{b_j Q_j}$}.
\end{enumerate}
\end{proposition}
\begin{proof}
Following the same line of arguments as in the proof of Theorem \ref{thm:1} we can prove, using  $\mathcal{S}$-Lemma, that   {\rm (I)} is equivalent to the semi-definite inequality:
\begin{equation}\label{eq:use1}
 \exists \, \lambda \ge 0, \mbox{ such that } \mymat{cc}{d_0-\vect{a}^T\vect{x}-\theta\vect{b}^T\vect{y}_0-\lambda r^2 & \ds\frac{1}{2}(\vect{d}^T-\vect{x}^T A-\theta\vect{b}^T W) \\ \ds\frac{1}{2}(\vect{d}^T-\vect{x}^T A-\theta\vect{b}^T W)^T & \lambda I_l -(1-\theta)\ds\sum_{j=1}^k{b_j Q_j}}  \succeq 0.
\end{equation}

{\bf [${\rm (I)} \Rightarrow {\rm (II)}$]}  We now show that \eqref{eq:use1} implies {\rm (II)}.
Observe that \eqref{eq:use1}  implies that, for each $p=1,\ldots,l$ the following $(2 \times 2)$ matrix
\[
\mymat{cc}{d_0-\vect{a}^T\vect{x}-\theta\vect{b}^T\vect{y}_0-\lambda r^2 & \ds\frac{1}{2}\left(\vect{d}-A^T\vect{x} - \theta W^T\vect{b}\right)_p \\ \ds\frac{1}{2}\left(\vect{d}-A^T\vect{x} - \theta W^T\vect{b}\right)_p & \lambda -(1-\theta)\sigma_p}  \succeq 0.
\]
So, $d_0-\vect{a}^T\vect{x}-\theta\vect{b}^T\vect{y}_0-\lambda r^2 \ge 0$, for each $p=1,\ldots,l,$
$\lambda  -(1-\theta)\sigma_p \ge 0$,
and
\begin{equation}\label{eq:00}
 \left(d_0-\vect{a}^T\vect{x}-\theta\vect{b}^T\vect{y}_0-\lambda r^2\right)\left(\lambda  -(1-\theta)\sigma_p\right) \ge \left[\ds\frac{1}{2}\left(\vect{d}-A^T\vect{x} - \theta W^T\vect{b}\right)_p\right]^2.
\end{equation}
Now, define the index $L$ by
\begin{equation}\label{def:L}
L=\{p \in \{1,\ldots,l\}: \lambda  -(1-\theta)\sigma_p> 0\},
\end{equation}
%
and let
\[
s_p=\left\{\begin{array}{ccc}
    0        & \mbox{ if } & p \notin L, \\
    \displaystyle \frac{\left[\ds\left(\vect{d}-A^T\vect{x} - \theta W^T\vect{b}\right)_p\right]^2}{4(\lambda  -(1-\theta)\sigma_p)}        & \mbox{ if } & p \in L.
           \end{array}
 \right.
\]
Then it follows that,  for all $p \notin L$, $ \lambda  -(1-\theta)\sigma_p=0$ and  \eqref{eq:00} gives us that
\[
\ds\left(\vect{d}-A^T\vect{x} - \theta W^T\vect{b}\right)_p =0.
\]
So, from the construction of $s_p$, we obtain that $s_p \ge 0$, $p=1,\ldots,l$, and,
\[
\left[\ds\left(\vect{d}-A^T\vect{x} - \theta W^T\vect{b}\right)_p\right]^2 \le 4 s_p(\lambda  -(1-\theta)\sigma_p), \ p=1,\ldots,l.
\]
Using the following well-known equivalence
 \begin{equation}\label{eq:trick}
t^2 \le 4 \alpha \beta, \ \alpha,\beta \ge 0 \  \Leftrightarrow \  \|(t,\alpha-\beta)\| \le \alpha+\beta,
\end{equation}
we obtain that
 $d_0-\vect{a}^T\vect{x}-\theta\vect{b}^T\vect{y}_0-\lambda r^2 \ge 0$, and for each $p=1,\ldots,l,$
$\lambda  -(1-\theta)\sigma_p \ge 0$,
and
\[
 \left\|\left(\left(\vect{d}-A^T\vect{x}^T -W^T\theta\vect{b}\right)_p,s_p-\lambda  +(1-\theta)\sigma_p \right) \right \| \le s_p +\lambda  -(1-\theta)\sigma_p.
\]
Let $M={\rm diag}(\lambda  -(1-\theta)\sigma_1,\ldots,\lambda  -(1-\theta)\sigma_l) \in \mathbb{R}^{l \times l}$ and  $u=\vect{d}-A^T\vect{x} - \theta W^T\vect{b} \in \mathbb{R}^l$.
For the  index $L \subseteq \{1,\ldots,l\}$ defined as before, let $M_L=(M_{\alpha \beta})_{\alpha,\beta \in L}$  and $u_L=(u_\alpha)_{\alpha\in L}$.
Then, \eqref{eq:use1} gives us that
\begin{equation}\label{eq:99}
\mymat{cc}{d_0-\vect{a}^T\vect{x}-\theta\vect{b}^T\vect{y}_0-\lambda r^2 & \frac{1}{2}u_L^T \\ \frac{1}{2} u_L & M_L }\succeq 0.
\end{equation}
Note from the definition of $L$ that $M_L \succ 0$. The Schur's complement together with \eqref{eq:99} implies that
\[
\displaystyle (d_0-\vect{a}^T\vect{x}-\theta\vect{b}^T\vect{y}_0-\lambda r^2) -\frac{1}{4} u_L^T M_L^{-1} u_L \ge 0.
\]
It follows from the definitions of $M_L$ and $u_L$  that
\begin{eqnarray*}
 0 &\le & (d_0-\vect{a}^T\vect{x}-\theta\vect{b}^T\vect{y}_0-\lambda r^2) -\frac{1}{4} \sum_{p \in L} \frac{[\left(\vect{d}-A^T\vect{x}^T -W^T\theta\vect{b}\right)_p]^2}{\lambda  -(1-\theta)\sigma_p} \\
 &=&   (d_0-\vect{a}^T\vect{x}-\theta\vect{b}^T\vect{y}_0-\lambda r^2) - \sum_{p \in L} s_p \\
 &=&   (d_0-\vect{a}^T\vect{x}-\theta\vect{b}^T\vect{y}_0-\lambda r^2) - \sum_{p=1}^l s_p,
\end{eqnarray*}
where the first equality is from the definition of $s_p$, $p=1,\ldots,l$, and the last system of equalities follows from the fact that $s_p=0$ for all $p \notin L$. So, {\rm (II)} holds.

{\bf [${\rm (II)} \Rightarrow {\rm (I)}$]} Suppose that {\rm (II)} holds. Define the index $L$ as in \eqref{def:L}. The last relation in {\rm (II)} shows that
\[
\left[\ds\left(\vect{d}-A^T\vect{x} - \theta W^T\vect{b}\right)_p\right]^2 \le 4 s_p \, \left(\lambda  -(1-\theta)\sigma_p\right), \ p=1,\ldots,l.
\]
So, $u_p=\ds\left(\vect{d}-A^T\vect{x} - \theta W^T\vect{b}\right)_p=0$ for all $p \notin L$, and for all $p \in L$
\[
s_p \ge \frac{\left[\ds\left(\vect{d}-A^T\vect{x} - \theta W^T\vect{b}\right)_p\right]^2}{4(\lambda  -(1-\theta)\sigma_p)}.
\]
This together with the second relation in  {\rm (II)} gives us that
\[
d_0-\vect{a}^T\vect{x}-\theta\vect{b}^T\vect{y}_0-\lambda r^2 \ge \sum_{p=1}^l s_p \ge \sum_{p \in L} s_p \ge \sum_{p \in L}\frac{\left[\ds\left(\vect{d}-A^T\vect{x} - \theta W^T\vect{b}\right)_p\right]^2}{4(\lambda  -(1-\theta)\sigma_p)}=\frac{1}{4} u_L^T M_L^{-1} u_L,
\]
where the first equality follows by noting that $s_p \ge 0$ for all $p=1,\ldots,l$, and the last equality follows from the definitions of $M_L$ and $u_L$.
This shows that \eqref{eq:99} holds. As for all $p \notin L$, $u_p=0$ and $\lambda  -(1-\theta)\sigma_p=0$, it follows that
\begin{equation}
\mymat{cc}{d_0-\vect{a}^T\vect{x}-\theta\vect{b}^T\vect{y}_0-\lambda r^2 & \frac{1}{2}u^T \\ \frac{1}{2} u & M }\succeq 0,
\end{equation}
and so, \eqref{eq:use1} holds. Hence,  {\rm (I)} follows.
\end{proof}

%


We now associate with $(P_s)$ the following second order cone program:

{\small \begin{eqnarray*}
({P_s{\emph-QDR}}) & \displaystyle \min_{\substack{\under{\vect{x}, \, \vect{y}_0,} \, W, \, \lambda_i, \, s_{p,i}, \, \sigma_{p,i}}} & \vect{c}^T\vect{x} \\
& \mbox{s.t.} &  \lambda_i\geq 0, \, s_{p,i} \ge 0, \; i=1,\dots,m, \, p=1,\ldots,l, \\
& &       \displaystyle \sum_{p=1}^l s_{p,i} \le d_{0,i}-\vect{a}^T\vect{x}-\theta\vect{b}_i^T\vect{y}_0-\lambda_i r^2, \; i=1,\dots,m, \\
 & &        \lambda_i  -(1-\theta)\sigma_{p,i} \ge 0, \ \ \ i=1,\ldots,m, \, p=1,\ldots,l,\\
  & &     \left\|\left(\left(\vect{d}_i-A^T\vect{x} -\theta W^T\vect{b}_i\right)_p,s_{p,i}-\lambda_i  +(1-\theta)\sigma_{p,i} \right) \right \| \le s_{p,i} +\lambda_i  -(1-\theta)\sigma_{p,i},\\
  & & \ \ \ \ \ \ \ \ \ \ \ \ i=1,\ldots,m, \ p=1,\ldots,l,
\end{eqnarray*}}
where $\vect{x}\in\R^{n},\;\vect{y}_0\in\R^{k},\;W\in\R^{k\times l}, \; \lambda_i \in \R, \; s_{p,i} \in \R,\; \sigma_{p,i}\in\R,\; p=1,\dots,l, \, i=1,\ldots,m$.


Using Proposition \ref{lemma:2.3}, we now show that the problem $(P_s)$ admits an exact SOCP reformulation in the sense that the objective values of $(P_s)$ and $(P_s\mbox{-QDR})$ are equal and their constraint systems are equivalent.

\begin{theorem}[{\bf Separable QDRs and Exact SOCP Reformulations}]\label{thm:2}
Let $\theta \in [0, 1]$. Consider the linear ARO problem $(P_s)$ with the parameterized separable quadratic decision rule and its associated second order cone program $(P_s\mbox{-QDR}) $.
Then, problem $(P_s)$ and the second order cone program $(P_s\mbox{-QDR})$ are equivalent, in
the sense that, $(\vect{x},\vect{y}_0,W,Q_1,\ldots,Q_k)$  is a solution for $(P_s)$ with $Q_j={\rm diag}(q_{1,j},\ldots,q_{l,j})$, $j=1,\ldots,k$, if and only if there exists $\lambda_i, \, s_{p,i}, \, \sigma_{p,i} \ge 0$, $p=1,\ldots,l$, $i=1,\ldots,m$, such that
$(\vect{x},\vect{y}_0, W, \lambda_i, \, s_{p,i}, \, \sigma_{p,i})$  is a solution for $(P_s\mbox{-QDR})$ with
$\sigma_{p,i}=\sum_{j=1}^k (\vect{b_i})_j q_{p,j}$, $p=1,\ldots,l$, $i=1,\ldots,m$. Moreover,
$
\min (P_s) =\min (P_s\mbox{-QDR})
$

\end{theorem}
\begin{proof} The constraint of $(P_s)$
\[
A(\vect{z})\vect{x} + B\left(\theta(\vect{y}_0 + W\vect{z}) + (1-\theta)\vect{z}^T \mathcal{Q}_k\vect{z}\right)\leq\vect{d}(\vect{z}),\quad \forall\vect{z}\in\ZZ
\]
can be equivalently rewritten as the following system of $m$ constraints:
\[
(\vect{a}_i + A_i\vect{z})^T\vect{x} + \vect{b}_i^T\left(\theta(\vect{y}_0 + W\vect{z}) + (1-\theta)\vect{z}^T \mathcal{Q}_k \vect{z}\right) \leq d_{0,i} + \vect{d}_i^T\vect{z},\;\forall\vect{z}\in\ZZ, \; i=1,2,\ldots, m.
\]
It now follows from Proposition \ref{lemma:2.3} that, for each $i=1,2,\ldots, m$, the system
\[
(\vect{a}_i + A_i\vect{z})^T\vect{x} + \vect{b}_i^T\left(\theta(\vect{y}_0 + W\vect{z}) + (1-\theta)\vect{z}^T \mathcal{Q}_k \vect{z}\right) \leq d_{0,i} + \vect{d}_i^T\vect{z},\;\forall\vect{z}\in\ZZ
\]
 {\small is equivalent to
\begin{eqnarray*}
& &\exists \lambda_i \geq 0, s_{p,i} \ge 0, \; p=1,\ldots,l \mbox{ such that} \\
& & \left\{
\begin{array}{l}
\lambda_i\geq 0, \, s_{p,i} \ge 0, \; i=1,\dots,m, \, p=1,\ldots,l, \\
       \displaystyle \sum_{p=1}^l s_{p,i} \le d_{0,i}-\vect{a}^T\vect{x}-\theta\vect{b}_i^T\vect{y}_0-\lambda_i r^2, \;  \\
          \lambda_i  -(1-\theta)\sigma_{p,i} \ge 0, \ \ \ i=1,\ldots,m, \, p=1,\ldots,l,\\
        \left\|\left(\left(\vect{d}_i-A^T\vect{x} -\theta W^T\vect{b}_i\right)_p,s_{p,i}-\lambda_i  +(1-\theta)\sigma_{p,i} \right) \right \| \le s_{p,i} +\lambda_i  -(1-\theta)\sigma_{p,i}, \; p=1,\ldots,l,
\end{array}
\right.
\end{eqnarray*}
where,} for each $i=1,\ldots,m$, $\sigma_{p,i}$, $p=1,\ldots,l$,  are the diagonal elements of $\sum_{j=1}^k{(\vect{b_i})_j Q_j}$, that is,
$\sigma_{p,i}=\sum_{j=1}^k (\vect{b_i})_j q_{p,j}$, $p=1,\ldots,l$, $i=1,\ldots,m$.
As the objective functions of both problems $(P_s)$ and $(P_s\mbox{-QDR})$ are the same, we obtain that $(\vect{x},\vect{y}_0,W,Q_1,\ldots,Q_k)$  is a solution for $(P_s)$ if and only if there exist $\lambda_i, \, s_{p,i}, \, \sigma_{p,i} \ge 0$, $p=1,\ldots,l$, $i=1,\ldots,m$, such that
$(\vect{x},\vect{y}_0,W, \lambda_i, \, s_{p,i}, \, \sigma_{p,i})$,   is a solution for $(P_s\mbox{-QDR})$ with
$\sigma_{p,i}=\sum_{j=1}^k (\vect{b_i})_j q_{p,j}$, $p=1,\ldots,l$, $i=1,\ldots,m$, and
$\min (P_s) =\min (P_s\mbox{-QDR}).$
\end{proof}

\begin{remark}
 Note that, in the SDP formulation, the decision variable is $(\vect{x},\vect{y}_0,W,Q_1,\ldots,Q_k)$
 which is of dimension $n+k+k l + k \frac{l(l+1)}{2}$; while in the SOCP formulation, the decision variable is $(\vect{x},\vect{y}_0, W, \lambda_i, \, s_{p,i}, \, \sigma_{p,i})$ which is of dimension
 $n+k+k l + m + 2ml$. 
 
\end{remark}

\section{ARO Problems with Objective and Constraint Adjustable Variables}
In this section we establish exact conic program reformulations for the following affinely parameterized version of ARO problem $(P_0)$  with adjustable variables also in the objective function: 
\begin{mini}
{\substack{\under{\vect{x},\vect{y}(\cdot)}}}{\vect{c}^T\vect{x}+ \max_{\vect{z}\in\ZZ} \{\vect{w}^T\vect{y}(z)\}}{}{\name{$\overline{P_0}$}}
	\addConstraint{A(\vect{z})\vect{x} + B\vect{y}(\vect{z})}{\leq\vect{d}(\vect{z}),\quad}{\forall\vect{z}\in\ZZ}
\end{mini}
where  $\ZZ = \left\{\vect{z}\in\R^l : \norm{\vect{z}}^2 \leq r^2\right\}$ is ellipsoidal uncertainty set; $\vect{c}\in\R^n$; $\vect{w} \in \mathbb{R}^k$; $B=(\vect{b}_1, \ldots, \vect{b}_m)^T, \vect{b}\in\R^k $;
$A(\vect{z}) = (\vect{a}_1+A_1\vect{z}, \dots, \vect{a}_m + A_m\vect{z})^T, \vect{a}_i\in\R^{n}, A_i\in\R^{n\times l}$ and  $\vect{d}(\vect{z}) = (d_{0,1}+\vect{d}_1^T\vect{z},\dots,d_{0,m} + \vect{d}_m^T\vect{z})^T,\; d_{0,i}\in\R,\; \vect{d}_i\in\R^{l}$. The problem $(\overline{P_0})$ with the quadratic decision rule as in Definition 2.1 takes the form: 
\begin{mini}
{\substack{\under{\vect{x},\vect{y}_0},\\W,Q_j,j=1,\dots,k}}{\vect{c}^T\vect{x}+ \max_{\vect{z}\in\ZZ} \{\vect{w}^T(\theta(\vect{y}_0 + W\vect{z}) + (1-\theta)\vect{z}^T\mathcal{Q}\vect{z})\}}{}{(\overline{P})}
	\addConstraint{A(\vect{z})\vect{x} + B\left(\theta(\vect{y}_0 + W\vect{z}) + (1-\theta)\vect{z}^T \mathcal{Q}_k \vect{z}\right)}{\leq\vect{d}(\vect{z}),\quad}{\forall\vect{z}\in\ZZ.}	
\end{mini}



We associate with $(\overline{P})$ the following semi-definite program

\begingroup
\begin{mini*}
{\substack{\under{\vect{x},\vect{y}_0, \vect{\lambda}, \tau},\\W,Q_j,j=1,\dots,k}}{\vect{c}^T\vect{x} + \tau}{}{({\overline{P}\mbox{-QDR}})}
	\addConstraint{ \lambda_i\geq 0, \; i=1,\dots,m}, \mymat{ccc}{
	P_1 & &  \\
	& \ddots \\
	& & P_{m+1}
	}\succeq 0,
\end{mini*}
where $\vect{x}\in\R^{n},\;\vect{y}_0\in\R^{k},\;\tau\in\R,\; W\in\R^{k\times l},\; Q_j\in\rS_l,j=1,\dots,k$ and
\[
P_i = \left\{\begin{array}{cl}
\mymat{cc}{d_{0,i}-\vect{a}_i^T\vect{x}-\theta\vect{b}_i^T\vect{y}_0-\lambda_i r^2 & \ds\frac{1}{2}(\vect{d}_i^T-\vect{x}^T A_i-\theta\vect{b}_i^T W) \\ \ds\frac{1}{2}(\vect{d}_i^T-\vect{x}^T A_i-\theta\vect{b}_i^T W)^T & \lambda_i I_l -(1-\theta)\ds\sum_{j=1}^k{(\vect{b}_i)_j Q_j}} & \quad i=1,\ldots,m, \\
 \mymat{cc}{\tau-\theta\vect{w}^T\vect{y}_0-\lambda_{m+1} r^2 & \ds\frac{1}{2}(-\theta\vect{w}^T W) \\ \ds\frac{1}{2}(-\theta\vect{w}^T W)^T & \lambda_i I_l -(1-\theta)\ds\sum_{j=1}^k{\vect{w}_j Q_j}} & \quad i=m+1.
\end{array}
  \right.P_d
\]
\endgroup

\begin{corollary}
 Let $\theta \in [0, 1]$. Consider the linear ARO problem $(\overline{P})$  with the parameterized quadratic decision rule and its associated semi-definite program $(\overline{P}\mbox{-QDR})$ . Then, problem $(\overline{P})$ and the semi-definite program $(\overline{P}\mbox{-QDR})$ are equivalent, in
the sense that, $(\vect{x},\vect{y}_0,W,Q_1,\ldots,Q_k)$  is a solution for $(\overline{P})$ if and only if there exists $\vect{\lambda} \in \mathbb{R}^m_+$ and $\tau\in\R$ such that
$(\vect{x},\vect{y}_0,\vect{\lambda},\tau,W,Q_1,\ldots,Q_k)$  is a solution for $(\overline{P}\mbox{-QDR})$. Moreover,
$
\min (\overline{P}) = \min (\overline{P}\mbox{-QDR})
$
\end{corollary}
\begin{proof} The problem $(\overline{P})$ can be equivalently rewritten as
 \begin{mini}
 {\substack{\under{\vect{x},\vect{y}_0},\\W,Q_j,j=1,\dots,k, \tau}}{\vect{c}^T\vect{x}+\tau}{}{}
	\addConstraint{A(\vect{z})\vect{x} + B\left(\theta(\vect{y}_0 + W\vect{z}) + (1-\theta)\vect{z}^T \mathcal{Q}_k \vect{z}\right)}{\leq\vect{d}(\vect{z}),\quad}{\forall\vect{z}\in\ZZ,}
	\addConstraint{ \vect{w}^T(\theta(\vect{y}_0 + W\vect{z}) + (1-\theta)\vect{z}^T\mathcal{Q}_k \vect{z}) \le \tau, \quad }{\forall \vect{z} \in \mathcal{Z},}
\end{mini}
The constraints of (3.2)
can equivalently be re-written as the following system of $m+1$ constraints:
\begin{equation}\label{eq:mconstraints}
 (\vect{a}_i + A_i\vect{z})^T\vect{x} + \vect{b}_i^T\left(\theta(\vect{y}_0 + W\vect{z}) + (1-\theta)\vect{z}^T \mathcal{Q}_k \vect{z}\right) \leq d_{0,i} + \vect{d}_i^T\vect{z},\;\forall\vect{z}\in\ZZ, \; i=1,2,\ldots, m
\end{equation}
and
\[
 \vect{w}^T \left(\theta(\vect{y}_0 + W\vect{z}) + (1-\theta)\vect{z}^T \mathcal{Q}_k \vect{z}\right) \le \tau.
\]
So, the conclusion follows by the same line of arguments as in Theorem \ref{thm:1}.
\end{proof}


Now, consider the following ARO problem with separable quadratic decision rule as in Definition \ref{def2}:
\begin{mini}
	{\substack{\under{\vect{x},\vect{y}_0} \\ W, Q_j, j\in [k]_+}}{\vect{c}^T\vect{x}+ \max_{\vect{z}\in\ZZ} \{\vect{w}^T(\theta(\vect{y}_0 + W\vect{z}) + (1-\theta)\vect{z}^T\mathcal{Q}_k \vect{z})\}}{}{(\overline{P}_s)}
	\addConstraint{A(\vect{z})\vect{x} + B\left(\theta(\vect{y}_0 + W\vect{z}) + (1-\theta)\vect{z}^T \mathcal{Q}_k \vect{z}\right)}{\leq\vect{d}(\vect{z}),\quad}{\forall\vect{z}\in\ZZ,}	
\end{mini}
where the assumptions of $(\overline{P}_s)$ are the same as of $(P_s)$. We associate with $(\overline{P}_s)$ the following second order cone program:
{\small \begin{eqnarray*}
({\overline{P}_s{\emph-QDR}}) & \displaystyle \min_{\substack{\under{\vect{x}, \, \vect{y}_0,} \, W, \, \tau, \\ \, \lambda_i, \, s_{p,i}, \, \sigma_{p,i}}} & \vect{c}^T\vect{x} +\tau \\
& \mbox{s.t.} &  \lambda_i\geq 0, \, s_{p,i} \ge 0, \; i=1,\dots,m+1, \, p=1,\ldots,l, \\
& &       \displaystyle \sum_{p=1}^l s_{p,i} \le d_{0,i}-\vect{a}^T\vect{x}-\theta\vect{b}_i^T\vect{y}_0-\lambda_i r^2, \; i=1,\dots,m, \\
 & &  \sum_{p=1}^l s_{p,m+1} \le \tau-\theta\vect{w}^T\vect{y}_0-\lambda_{m+1} r^2, \\
 & &        \lambda_i  -(1-\theta)\sigma_{p,i} \ge 0, \ \ \ i=1,\ldots,m+1, \, p=1,\ldots,l,\\
  & &     \left\|\left(\left(\vect{d}_i-A^T\vect{x} -\theta W^T\vect{b}_i\right)_p,s_{p,i}-\lambda_i  +(1-\theta)\sigma_{p,i} \right) \right \|  \le s_{p,i} +\lambda_i  -(1-\theta)\sigma_{p,i}, \\
& & \ \ \ \ \ \ \ \ \ \ \ \ i=1,\ldots,m, \ p=1,\ldots,l,  \\
  & & \left\|\left((-\theta W^T\vect{w})_p,s_{p,m+1}-\lambda_{m+1}  +(1-\theta)\sigma_{p,m+1} \right) \right \| \le s_{p,m+1} +\lambda_{m+1}  -(1-\theta)\sigma_{p,m+1}, \\
  & & \ \ \ \ \ \ \ \ \ \ \ \ p=1,\ldots,l.
\end{eqnarray*}}

\begin{corollary}
 Let $\theta \in [0, 1]$. Consider the linear ARO problem  $(\overline{P})$ with the parameterized separable quadratic decision rule and its associated second order cone program $(\overline{P}_s\mbox{-QDR})$. Then, problem $(\overline{P})$ and the second order cone program
  $(\overline{P}_s\mbox{-QDR})$ are equivalent, in
the sense that,
$(\vect{x},\vect{y}_0,W,Q_1,\ldots,Q_k)$  is a solution for $(\overline{P})$ with $Q_j={\rm diag}(q_{1,j},\ldots,q_{l,j})$, $j=1,\ldots,k$, if and only if there exists $\tau,\, \lambda_i, \, s_{p,i}, \, \sigma_{p,i} \ge 0$, $p=1,\ldots,l$, $i=1,\ldots,m+1$, such that
$(\vect{x},\vect{y}_0, W, \tau, \, \lambda_i, \, s_{p,i}, \, \sigma_{p,i})$  is a solution for $(\overline{P}_s\mbox{-QDR})$ with
$\sigma_{p,i}=\sum_{j=1}^k (\vect{b_i})_j q_{p,j}$, $p=1,\ldots,l$, $i=1,\ldots,m$ and $\sigma_{p,m+1}=\sum_{j=1}^k w_j q_{p,j}$, $p=1,\ldots,l$. Moreover,
$
\min (\overline{P}_s) = \min (\overline{P}_s\mbox{-QDR}).
$
\end{corollary}
\begin{proof}
As we have seen in Corollary 2.8,  the problem $(\overline{P}_s)$ can be equivalently rewritten as
 \begin{mini}
	{\substack{\under{\vect{x},\vect{y}_0, W, \tau} \\ Q_j, j=1,\dots,k}}{\vect{c}^T\vect{x} + \tau  }{}{}
	\addConstraint{A(\vect{z})\vect{x} + B(\theta(\vect{y}_0 + W\vect{z}) + (1-\theta)\vect{z}^T\mathcal{Q}_k \vect{z})(\vect{z})}{\leq\vect{d}(\vect{z}),\quad}{\forall\vect{z}\in\ZZ}
	\addConstraint{ \vect{w}^T(\theta(\vect{y}_0 + W\vect{z}) + (1-\theta)\vect{z}^T\mathcal{Q}_k \vect{z}) \le \tau, \quad }{\forall \vect{z} \in \mathcal{Z}.}
\end{mini}
Now, the conclusion follows from Proposition \ref{lemma:2.3}. \end{proof}

%
%

\begin{remark}
 We note that, our exact SDP (resp. SOCP reformulation) continues to hold in the general case where the cost vector $\vect{c}$ is also uncertain and it belongs to the norm uncertainty set $\mathcal{U}:=\{\vect{c}: \|\vect{c}-\vect{c}_0\| \le \rho\}$ for some
 $\vect{c}_0 \in \mathbb{R}^n$ and $\rho \ge 0$. Here $\|\cdot \|$ denotes a norm in $\mathbb{R}^n$. Indeed, in this case, problem $(\overline{P})$ becomes
 \begin{mini}
	{\substack{\under{\vect{x},\vect{y}_0} \\ W,Q_j,j=1,\dots,k}}{ \max_{\vect{c} \in \mathcal{U}}\vect{c}^T\vect{x}+ \max_{\vect{z}\in\ZZ} \{\vect{w}^T(\theta(\vect{y}_0 + W\vect{z}) + (1-\theta)\vect{z}^T\mathcal{Q}_k \vect{z})\}}{}{}
	\addConstraint{A(\vect{z})\vect{x} + B(\theta(\vect{y}_0 + W\vect{z}) + (1-\theta)\vect{z}^T\mathcal{Q}_k \vect{z})}{\leq\vect{d}(\vect{z}),\quad}{\forall\vect{z}\in\ZZ,}
\end{mini}
which can be further rewritten as
 \begin{mini}
	{\substack{\under{\vect{x},\vect{y}_0, \tau_1,\tau_2}\\ W, Q_j, j=1,\dots,k}}{\vect{c}^T\vect{x} + \tau_1+\tau_2  }{}{}
	\addConstraint{A(\vect{z})\vect{x} + B(\theta(\vect{y}_0 + W\vect{z}) + (1-\theta)\vect{z}^T\mathcal{Q}_k \vect{z})}{\leq\vect{d}(\vect{z}),\quad}{\forall\vect{z}\in\ZZ}
	\addConstraint{ \vect{w}^T(\theta(\vect{y}_0 + W\vect{z}) + (1-\theta)\vect{z}^T\mathcal{Q}_k \vect{z}) \le \tau_1, \quad }{\forall \vect{z} \in \mathcal{Z}}
		\addConstraint{  \vect{c}_0^T\vect{x} +\rho \|\vect{x}\| \le \tau_2. }{ }
\end{mini}
Thus, the conclusion follows by employing the same line of arguments as in the proof of the preceding two corollaries.
\end{remark}


 \section{Lot-Sizing Problem: Worst-case \& Uncertainty-Realisation Comparisons}
In the lot-sizing problem on a network, we consider $N$ stores, for which stock allocations must be determined to fulfill the demand at each store. Stock can be delivered at the beginning of the day and stored, or transported from another store at a later point in time. Let $x_i$ denote the quantity of stock to initially deliver to store $i$, with unit storage cost $c_i$. Each store can hold up to $\Gamma$ units of stock at any time. Let $y_{ij}$ denote the quantity of stock to transport from store $i$ to store $j$, with unit transportation cost $t_{ij}$. Note that $t_{ii} = 0$ and $t_{ij}$ is not necessarily equal to $t_{ji}$.

In general the demand for store $i$, denoted $z_i$, is uncertain at the beginning of the day, only known to reside in some uncertainty set $\ZZ$. Hence, we formulate the problem as a two-stage adjustable robust problem, by allowing the transportation decisions $y_{ij}$ to become wait-and-see variables. That is, an initial stock delivery $\vect{x}$ is sent to all stores at the beginning of the day, and once the demand $\vect{z}$ is revealed, the transportation decisions $y_{ij}(\vect{z}),\; i,j=1,\dots,N$ are implemented to fulfill the demand at each store. Wanting to minimize total costs, $\sum_{i=1}^{N}{c_ix_i} + \sum_{i,j=1}^{N}{t_{ij}y_{ij}(\vect{z})}$, this gives the following ARO formulation:
\begin{mini*}
	{\substack{\under{\vect{x}\in \mathcal{X}\subset\R^N, \tau\in\R,}\\\under{y_{ij}:\ZZ\subseteq\R^N\rightarrow\R,}\\i,j=1,\dots,N}}{\sum_{i=1}^{N}{c_i x_i} + \max_{\under{\vect{z}}\in\ZZ}{\left\{\sum_{i,j=1}^{N}{t_{ij}y_{ij}(\vect{z})}\right\}}}{}{\text{(LS)}}
	\addConstraint{x_i+\sum_{j=1}^{N}{y_{ji}(\vect{z})}-\sum_{j=1}^{N}{y_{ij}(\vect{z})}}{\geq z_i,\quad}{\forall \vect{z}\in\ZZ,,\quad i=1\dots N}
	\addConstraint{y_{ij}(\vect{z})}{\geq 0,\quad}{\forall \vect{z}\in\ZZ,\quad i,j=1\dots N,}
\end{mini*}
where $x_i$ are the here-and-now decisions, $\vect{x}\in\mathcal{X} = \{\vect{x}\in\R^{N} : 0\leq x_i\leq \Gamma,\; i=1,\dots,N\}$, $y_{ij}$ are the wait-and-see variables, with uncertainty set $\ZZ = \{\vect{z}\in\R^{N} : \norm{\vect{z}}^2\leq\frac{\Gamma^2}{2}\}$ (for further details, see \cite{Zhen-Motzkin-18}).

We wish to compare the solution methods of direct ADR substitution, QDR via SDP as in Corollary 3.1 and Separable QDR via SOCP as in Corollary 3.2. We will first compare their solutions after realisation of the uncertain demand $\vect{d}$, by comparison of the realised cost to the true solution given by
\begin{mini*}
	{\under{\vect{x}, y_{ij}}}{\sum_{i=1}^{N}{c_i x_i} + \sum_{i,j=1}^{N}{t_{ij} y_{ij}}}{}{\text{(TD)}\quad}
	\addConstraint{x_i + \sum_{j=1}^{N}{y_{ji}} - \sum_{j=1}^{N}{y_{ij}} \geq d_i,\quad i=1,\dots,N}
	\addConstraint{0 \leq x_i \leq \Gamma,\quad i=1,\dots,N}
	\addConstraint{y_{ij}\geq 0,\quad i,j=1,\dots,N.}
\end{mini*}

We will then compare our solution methods to the worst-case solution, given by substitution into (TD) of the worst case value for $d_i$ in $\ZZ$; namely, $d_i = \frac{\Gamma}{\sqrt{2}}$:
\begin{mini*}
	{\under{\vect{x}, y_{ij}}}{\sum_{i=1}^{N}{c_i x_i} + \sum_{i,j=1}^{N}{t_{ij} y_{ij}}}{}{\text{(WC)}\quad}
	\addConstraint{x_i + \sum_{j=1}^{N}{y_{ji}} - \sum_{j=1}^{N}{y_{ij}} \geq \frac{\Gamma}{\sqrt{2}},\quad i=1,\dots,N}
	\addConstraint{0 \leq x_i \leq \Gamma,\quad i=1,\dots,N}
	\addConstraint{y_{ij}\geq 0,\quad i,j=1,\dots,N.}
\end{mini*}

Note that the direct ADR substitution into (LS) is solved via a SOCP (see, e.g. \cite[Theorem 3.1]{7}). 

We create 50 random instances of the lot-sizing problem by generating storage and transportation costs from the uniform distribution on $[0, 1000]$. We produce a random demand $\vect{d}\in\ZZ$ for each. We then compare the methods by calculating the following \emph{percentage difference} metrics:
\begin{align*}
& m_1 = 100\cdot\frac{v-t}{v}, && m_2 = 100\cdot\frac{w-v}{w}
\end{align*}
where $v$ is the optimal (realized/worst-case) value produced by solving (LS) via the method, $t$ is the optimal value for (TD), $w$ is the optimal value for (WC), $m_1$ is a comparison metric against (TD) and $m_2$ is a comparison metric against the (WC). We also compute the average time taken to solve a single problem instance, for each of these methods. Note that the lower the calculated $m_1$ and the higher the calculated $m_2$, the better the performance of the method.

All computations were performed using a 3.2GHz Intel(R) Core(TM) i7-8700 and 16GB of RAM, equipped with MATLAB R2019B. All problem instances, being conic programs, were solved using the CVX toolbox (see, e.g. \cite{CVX-sofware14}).

\begin{table}[H]
\centering
\captionsetup{font = footnotesize}
\begin{tabu}{| l |c | c | c |}
\hline
\textbf{True Solution} & ADR via SOCP \cite{7}. & QDR via SDP & Separable QDR via SOCP\Tstrut\Bstrut\\
\hline
\hline
$N = 2$: \% Diff. & 67.7072  & 64.7297  & 64.9984	\Tstrut\Bstrut\\
\hline
\hline
$N = 3$: \% Diff. & 67.5402  & 63.9634  & 64.4285\Tstrut\Bstrut\\
\hline
\hline
$N = 4$: \% Diff. & 70.6617  & 65.9177  & 66.6341\Tstrut\Bstrut\\
\hline
\hline
$N = 5$: \% Diff. & 68.8926  & 63.6860  & 64.5545\Tstrut\Bstrut\\
\hline
\hline
$N = 8$: \% Diff. & 71.8619  & 64.1862  & 65.3614\Tstrut\Bstrut\\
\hline
\end{tabu}
\caption{
Results for the case of true solution comparison. \% Diff. represents the average percentage difference between the solution to (TD) and the realised solution for the method ($m_1$). Time is measured in seconds.
}
\end{table}

\begin{figure}[hbt!]
\centering
\begin{minipage}{.5\textwidth}
  \centering
  \includegraphics[width=\linewidth]{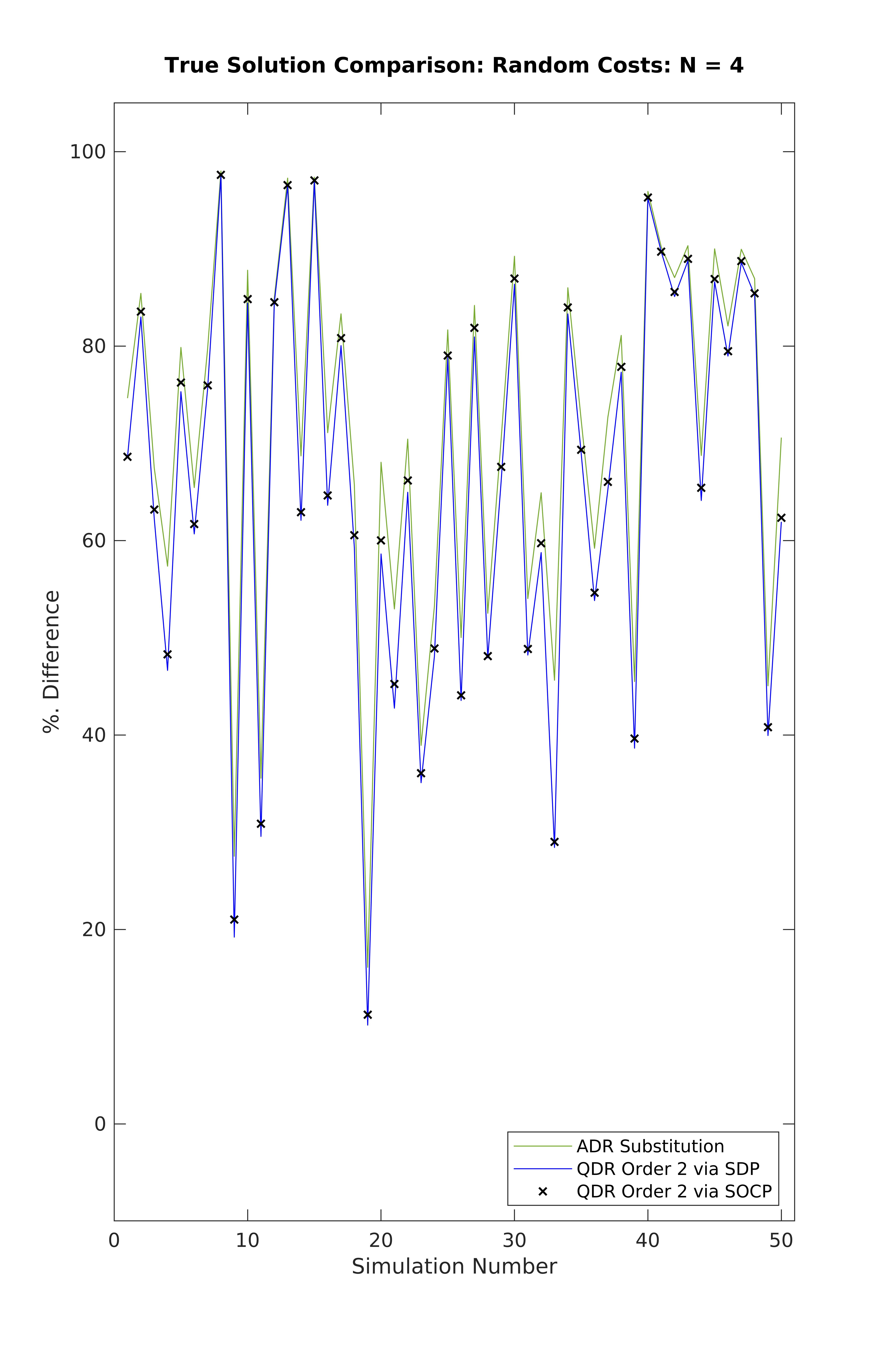}
\end{minipage}%
\begin{minipage}{.5\textwidth}
  \centering
  \includegraphics[width=\linewidth]{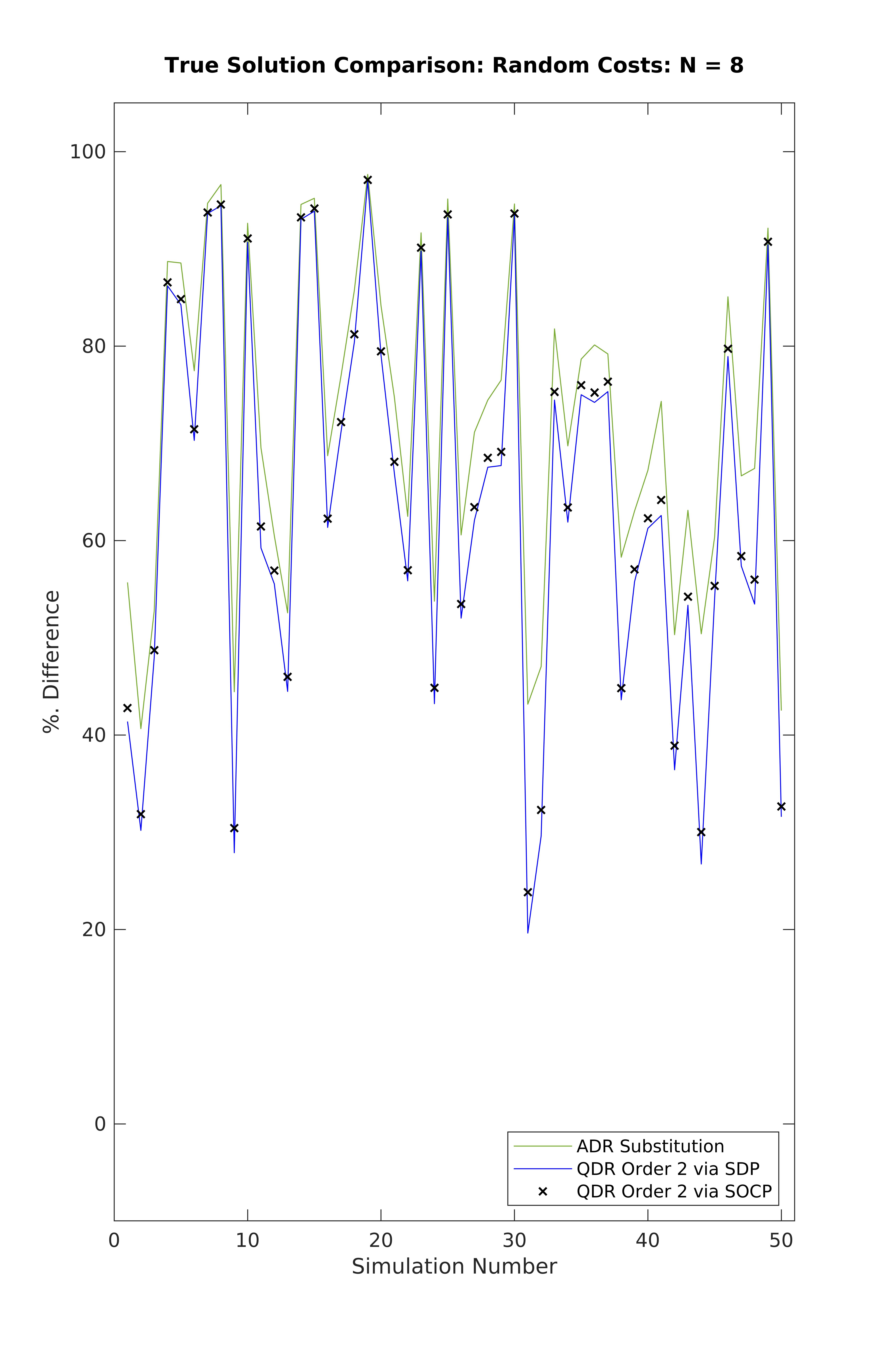}
\end{minipage}
\end{figure}

\clearpage

\begin{table}[H]
\centering
\captionsetup{font = footnotesize}
\begin{tabu}{| l |c | c | c |}
\hline
\textbf{Worst-Case} & ADR via SOCP \cite{7}. & QDR via SDP  & Separable QDR via SOCP\Tstrut\Bstrut\\
\hline
\hline
$N = 2$: \% Diff. & 17.2541  & 20.7104  & 20.6353	\Tstrut\Bstrut\\
\hline
\hline
$N = 3$: \% Diff. & 30.3298  & 36.4901  & 35.6099\Tstrut\Bstrut\\
\hline
\hline
$N = 4$: \% Diff. & 42.1128  & 49.4073 & 48.2685\Tstrut\Bstrut\\
\hline
\hline
$N = 5$: \% Diff. & 46.9065  & 55.1508  & 53.9234\Tstrut\Bstrut\\
\hline
\hline
$N = 8$: \% Diff. & 63.2694  & 71.2344  & 70.2952\Tstrut\Bstrut\\
\hline
\end{tabu}
\caption{
Results for the case of worst-case comparison and random costs. \% Diff. represents the average percentage difference between worst-case solution of (WC) and the worst-case solution for the method ($m_2$). Average Time is not presented as it is previously demonstrated in Table 1.
}
\end{table}

\begin{figure}[h]
\centering
\begin{minipage}{.5\textwidth}
  \centering
  \includegraphics[width=\linewidth]{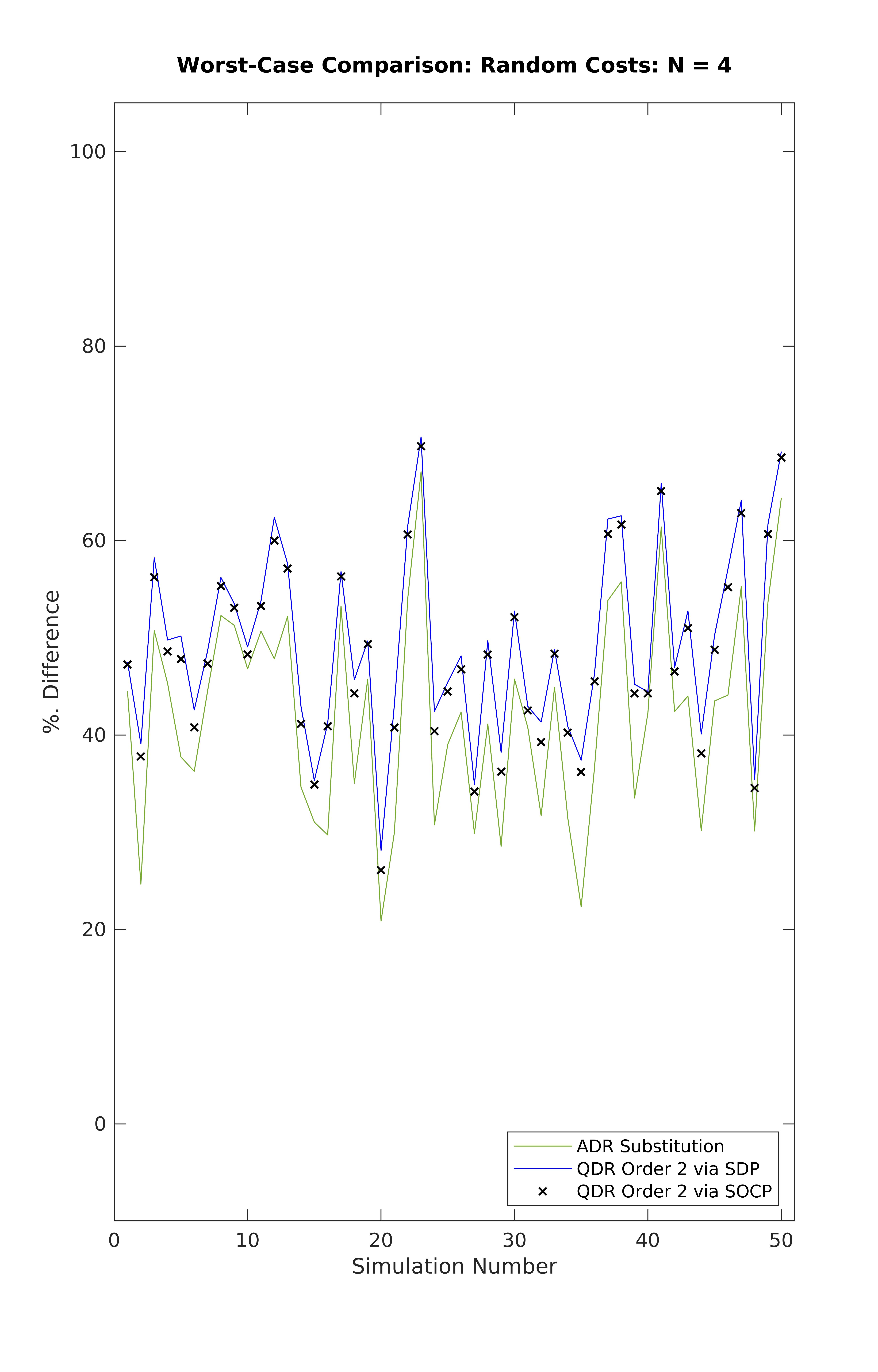}
\end{minipage}%
\begin{minipage}{.5\textwidth}
  \centering
  \includegraphics[width=\linewidth]{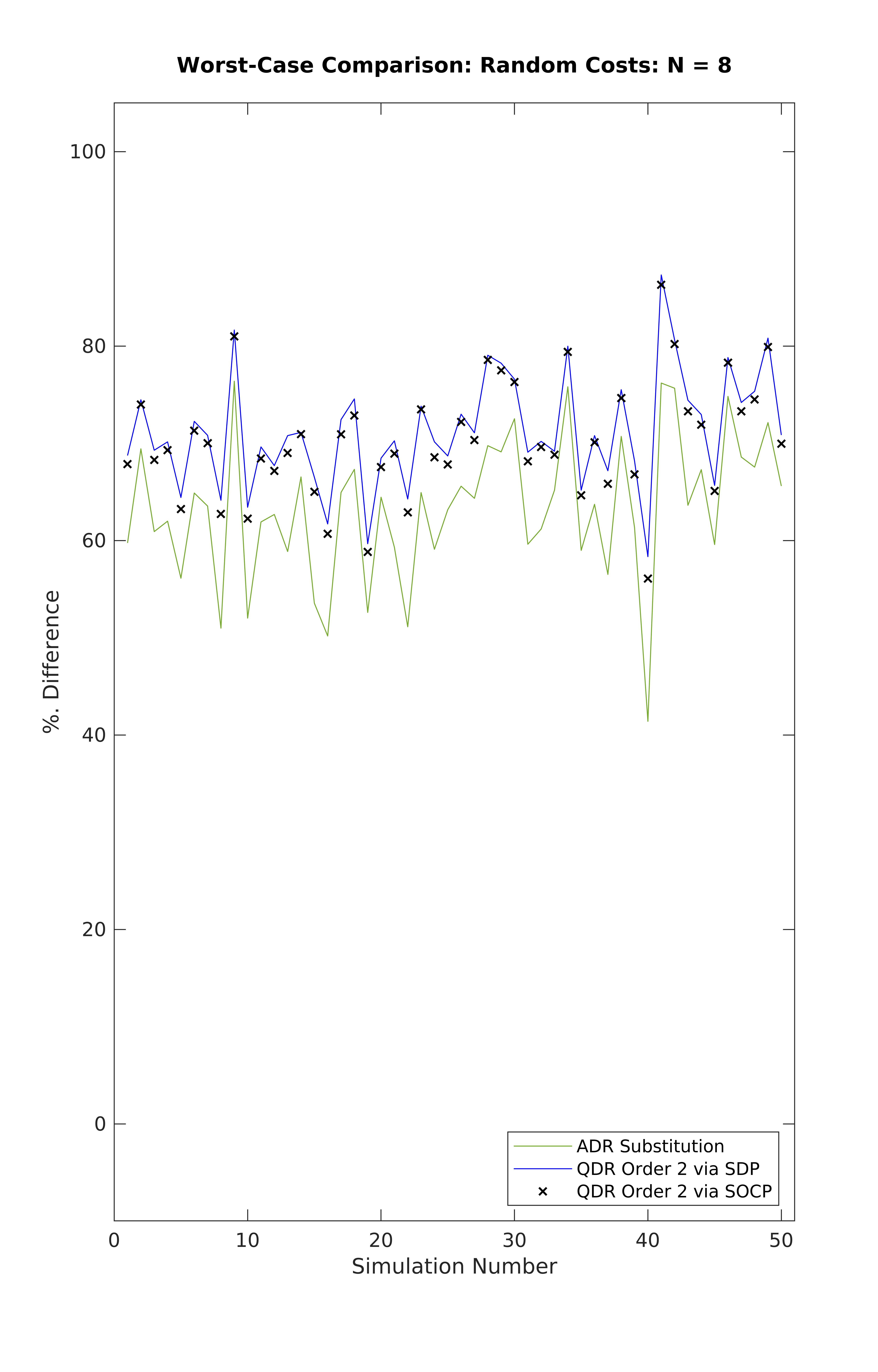}
\end{minipage}
\end{figure}

\clearpage

\begin{table}[H]
\centering
\captionsetup{font = footnotesize}
\begin{tabu}{| l |c | c | c |}
\hline
\textbf{Problem Sizes} & ADR via SOCP \cite{7}. & QDR via SDP  & Separable QDR via SOCP\Tstrut\Bstrut\\
\hline
\hline
\,\,$N = 2$: Variables: & 25  & 78  & 154 \Tstrut\Bstrut\\
\hline
\rowfont{\color{blue}}
Constraints: & 14 &  27 &  73\Tstrut\Bstrut \\
\hline
\hline
\,\,$N = 3$: Variables: & 63  & 200  & 370 \Tstrut\Bstrut\\
\hline
\rowfont{\color{blue}}
Constraints: & 25  & 84  & 186\Tstrut\Bstrut \\
\hline
\hline
\,\,$N = 4$: Variables: & 123  & 435  & 728 \Tstrut\Bstrut\\
\hline
\rowfont{\color{blue}}
Constraints: & 41  & 215  & 383 \Tstrut\Bstrut \\
\hline
\hline
\,\,$N = 5$: Variables: & 213  & 840  & 1264 \Tstrut\Bstrut\\
\hline
\rowfont{\color{blue}}
Constraints: & 61  & 468  & 688 \Tstrut\Bstrut \\
\hline
\hline
\,\,$N = 8$: Variables: & 723  & 3890  & 4813 \Tstrut\Bstrut\\
\hline
\rowfont{\color{blue}}
Constraints: & 145  & 1271  & 2322 \Tstrut\Bstrut \\
\hline
\end{tabu}
\caption{
Representation of the problem size for each method. Number of optimisation variables and number of equality constraints as outputted by CVX after solving.
}
\end{table}


%
%
%
%

Based on the numerical experiments, we can conclude then that both SDP and SOCP reformulation based methods for solving affinely parameterized ARO problems with a quadratic decision rule exceed the performance of the classical ADR approach.

With our state-of-the-art conic programming solver, we were only able to solve up to size $N=8$, as is demonstrated in Table 1. 



\section{Conclusion and Outlook}

In this paper we have shown that affinely parameterized linear adjustable robust optimization problems with a new parametric QDRs under ellipsoidal uncertainty are numerically tractable by establishing exact  semi-definite program and second order cone program reformulations.  We have also demonstrated via numerical experiments on  lot-sizing problems with uncertain demand that  these adjustable robust linear optimization problems with QDRs improve upon the affine decision rules in their performance both in the worst-case sense and after simulated realization of the uncertain demand relative to the true solution. It is of great interest to study computational tractability of adjustable robust linear optimization problems with QDRs in the presence of uncertainty sets that are expressed as the intersection of ellipsoids and will be examined in a forthcoming study.

\end{document}